\documentclass[11pt]{article}
\usepackage{epsfig,amsmath,amsthm,latexsym,graphicx,amsfonts,amssymb}
\usepackage[numbers]{natbib}
\newif\ifiota
\iotatrue
\newif\ifZ
\Ztrue
\newcommand{\beq}{\begin{equation}}
\newcommand{\eeq}{\end{equation}}
\newcommand{\R}{\mathbb R}

\newcommand{\cX}{\mathcal{X}}
\newcommand{\cY}{\mathcal{Y}}
\newcommand{\cZ}{\mathcal{Z}}
\newcommand{\cM}{\mathcal{M}}
\newcommand{\cN}{\mathcal{N}}

\newcommand{\cT}{\mathcal{T}}

\newcommand{\gk}{\hat{g}}
\newcommand{\hk}{\hat{h}}
\newcommand{\gkn}{\gk_n}
\newcommand{\gkone}{\gk_1}
\newcommand{\hkn}{\hk_n}
\newcommand{\hkone}{\hk_1}

\newcommand{\meansuff}{\tau_\theta}

\newcommand{\cartMn}{\prod^n \cM}
\newtheorem{theorem}{Theorem}

\newtheorem{remark}{Remark}
\newtheorem*{assumption}{Assumptions}

\begin{document}

\title{Chentsov's theorem for exponential families}
\author{James G. Dowty}
\date{\today}

\maketitle

\abstract{
Chentsov's theorem 
characterizes the Fisher information metric on statistical models as essentially the only Riemannian metric that is invariant under sufficient statistics.  This implies that each statistical model is naturally equipped with a geometry, so Chentsov's theorem explains
why many statistical properties can be described in geometric terms.
However, despite being one of the foundational theorems of statistics, Chentsov's theorem has only been proved previously in very restricted settings or under relatively strong regularity and invariance assumptions.  We therefore prove a version of this theorem for the 
important case of exponential families.
In particular, we characterise the Fisher information metric as the only Riemannian metric (up to rescaling) on an exponential family and its derived families that is invariant under independent and identically distributed extensions and canonical sufficient statistics.
Our approach is based on the central limit theorem, so it gives a unified proof for both discrete and continuous exponential families, and it is less technical than previous approaches.
}

\section{Introduction}
\label{S:intro}

Chentsov's theorem is a foundational theorem in statistics that characterizes the Fisher information metric on statistical models as the only Riemannian metric (up to rescaling) that is invariant under certain, statistically important transformations \citep{Chentsov78,Chentsov82,Campbell86,AyEtAl15,BauerEtAl16}.
This effectively means that the Fisher information metric is the only natural metric on a statistical model, so
many statistical properties of these models should be
describable in terms of this metric.
Known examples of this correspondence between statistical and geometric properties include:  the Cram\'er-Rao lower bound for the variance of an unbiased estimator
in terms of the inverse of the Fisher information metric
\citep[Thm. 2.2]{AmariNagaoka00};
orthogonality as a criterion for first-order efficiency of estimators \citep[Thm. 4.3]{AmariNagaoka00}; the central role of statistical curvature in the information loss of an efficient estimator \citep[\S 3.3]{KassVos97} and in second-order efficiency \citep[\S 3.4]{KassVos97}; and the spontaneous emergence of the Fisher information volume \citep{Rissanen96} in the minimum description length (MDL) approach to statistical model selection \citep{BarronEtAl98}.

The original version of Chentsov's theorem \citep{Chentsov78,Chentsov82,Campbell86} only applied in the restricted setting of statistical models with finite data spaces.
This version of the theorem says that
the Fisher information metric is the only metric (up to a multiplicative constant) that is defined on all models with finite data spaces and is invariant under all sufficient statistics.
Recall that a statistical model $\cM$ is a (sufficiently regular) set of probability measures on the same measurable space $\cX$, which we call the data space of $\cM$, and that a sufficient statistic for $\cM$ is a function on $\cX$ for which the conditional distribution of any measure $P$ in $\cM$, given the sufficient statistic, is the same for all $P$.  Sufficient statistics induce corresponding maps on statistical models (the measure-theoretic push-forward maps) and
the invariance assumption above is that all of these maps are isometries (i.e., distance-preserving maps).

Since the assumption of finite data spaces is very restrictive, \citet{AyEtAl15} proved a version of Chentsov's theorem that applies to models whose data spaces $\cX$ are smooth manifolds.  Their version says that the Fisher information metric is the only metric (up to rescaling) that is defined on all statistical models with a given data space $\cX$ and is invariant under all sufficient statistics, including discontinuous ones.  This version of Chentsov's theorem applies to many interesting statistical models but it makes very strong assumptions about
both the breadth of the models on which the metrics are defined and
the invariance properties of these metrics.
Therefore \citet{BauerEtAl16} proved a version of Chentsov's theorem which says the Fisher information metric is the only metric (up to rescaling) that firstly is defined on the space of all smooth, positive densities on a compact manifold $\cX$ of dimension $2$ or higher and secondly is invariant under all diffeomorphisms from $\cX$ to itself (where diffeomorphisms are smooth maps with smooth inverses, so they are a special type of sufficient statistic).  The proof of \citet{BauerEtAl16} was based on results from the theory of generalized functions, especially the Schwartz kernel theorem \citep[\S 6.1]{FriedlanderJoshi98}, and it made far weaker invariance assumptions than that of \citet{AyEtAl15}.
The assumption that $\cX$ is a compact manifold without boundary excludes many cases of interest to statisticians, though \citet{BauerEtAl16} say this assumption can be weakened.

Despite their beauty and generality, the results of \citet{AyEtAl15} and \citet{BauerEtAl16} leave open the possibility that there might exist a natural metric other than the Fisher information metric on an individual statistical model $\cM$. This could occur, for example, if there is a natural metric on $\cM$ that does not (invariantly) extend to a metric on the infinite-dimensional models of \citep{AyEtAl15} and \citep{BauerEtAl16} that contain $\cM$ and many unrelated models.
Also, exponential families have a distinguished, finite-dimensional set of sufficient statistics, called the canonical sufficient statistics, which are related to their natural affine structures (\citep[Thm. 2.4]{AmariNagaoka00} and \citep[Lemma 8.1]{BarndorffNielsen78}).
Therefore, the invariance assumptions of \citep{AyEtAl15} and \citep{BauerEtAl16}
are arguably too strong for exponential families, and instead it would be more natural to consider invariance under canonical sufficient statistics rather than all sufficient statistics.

In this paper, we prove a refined version of Chentsov's theorem in the important case of exponential families.  Instead of considering metrics defined on an infinite-dimensional statistical model, as in \citep{AyEtAl15} and \citep{BauerEtAl16}, we consider metrics defined only on
a given exponential family $\cM$ and some of its derived families, namely its independent and identically distributed (IID) extensions and their corresponding natural exponential families.  Instead of assuming these metrics are invariant under all sufficient statistics or all diffeomorphisms, we assume invariance under canonical sufficient statistics and IID extensions.  This assumption of invariance under IID extensions has no analogue in previous work, but
IID extensions are natural and important transformations between statistical models (perhaps more so than sufficient statistics), so this invariance assumption is arguably more natural than invariance under sufficient statistics.  Also, this extra invariance assumption is offset by the fact that we restrict our sufficient statistics to the canonical ones.
Then, under a mild regularity condition,
we prove that metrics with these invariance properties are multiples of the Fisher information metric (see Theorem \ref{T:main} in Section \ref{S:chentsov}).  This result therefore gives a new characterisation of the Fisher information metric as the only metric on an exponential family and its derived families that is invariant under canonical sufficient statistics and IID extensions.

Our approach has a number of advantages:  as discussed above, we only assume that the metric is defined on an individual model and its related models, and our invariance assumptions respect the natural affine structures of exponential families; we only consider metrics on a collection of finite-dimensional models (similar to the original version of Chentsov's theorem \citep{Chentsov78,Chentsov82,Campbell86}), which allows us to avoid the technicalities encountered in \citep{AyEtAl15} and \citep{BauerEtAl16} because of the infinite-dimensionality of their statistical models; our proof is unified for discrete and continuous distributions, unlike the proofs of \citep{Chentsov78,Chentsov82,Campbell86} and \citep{BauerEtAl16}, so there is some hope of extending our proof to general statistical models; our proof shows that Chentsov's theorem is a corollary of the central limit theorem, which makes this result more understandable and intuitive; and our results complement those of \citep{BauerEtAl16}, since (curved) exponential families are essentially the only statistical models with smooth sufficient statistics that are not diffeomorphisms,
by the  Pitman--Koopman--Darmois theorem \citep{BarndorffNielsenPedersen68}.

The rest of this paper is set out as follows.  In Section \ref{S:FImetric} we define the Fisher information metric and some relevant notions from differential geometry, as they apply in our main case of interest.  In Section \ref{S:expfam} we briefly recall the definition of an exponential family and some of its derived families.  We then give precise descriptions of our assumptions in Section \ref{S:assumptions}, before using these assumptions and the central limit theorem to prove our characterisation of the Fisher information metric in Section \ref{S:chentsov}.  We then describe an extension of our proof to higher-order symmetric tensors in Section \ref{S:higherorder}, before finishing with a discussion of our results in Section \ref{S:conclusion}.  Section \ref{S:conclusion} also begins with a non-technical summary of our proof.

\section{The Fisher information metric}
\label{S:FImetric}

This section briefly recalls the definitions of tangent vectors and the Fisher information metric of a statistical model.  A general reference for the notions from Riemannian geometry described here is \cite[Appendix C]{KassVos97}.

In all later sections of this paper, we will take $\cM$ to be a regular exponential family with natural parameter space $\Theta$, but in this section we let $\cM$ be a more general statistical model and let $\Theta$ be any smooth parameter space for $\cM$.  More precisely, suppose $\Theta$ is an open subset of $\R^d$ and that $\mu$ is a measure on $\R^m$ with support $\cX$.  Then our statistical model is $\cM = \{ p_\theta \mu \mid \theta \in \Theta \}$, where each $p_\theta: \cX \to \R_{>0}$ is a $\mu$-integrable, strictly positive function that is normalized, meaning $1 = \int p_\theta d\mu$.  Note that $\cM$ is a set of probability measures on $\R^m$.
We assume that the parameterisation of $\cM$ by $\Theta$ is smooth, in the sense that $\theta \mapsto p_\theta(x)$ is a smooth (i.e., infinitely differentiable) function for $\mu$-almost all $x$.
We also assume that the parameterisation is non-singular, meaning that the parameterisation map $\Theta \to \cM$ given by $\theta \mapsto p_\theta \mu$ is injective and that it maps non-zero tangent vectors to non-zero tangent vectors, in a sense which will become clear below.

Because $\Theta$ is an open subset of $\R^d$, any tangent vector $u$ to $\Theta$ is a pair $u=(\theta,a)$ for some $\theta \in \Theta$ and some $a \in \R^d$, where $\theta$ is called the base-point of $u$.  The set of all such tangent vectors, which is denoted $T\Theta$ and is called the tangent bundle of $\Theta$, is therefore $T\Theta = \Theta \times \R^d$.
The tangent bundle is not a vector space in general, but the set of all tangent vectors with the same base-point is.  The vector space $T_\theta \Theta$ consisting of all vectors with base-point $\theta$ is called the tangent space to $\Theta$ at $\theta$.
Addition and scalar multiplication in this vector space are given by
\begin{align}
\label{E:vecsp}
su + tv = (\theta,sa+tb)
\end{align}
for any $u, v \in T_\theta \Theta$ and any $s, t \in \R$, where $u=(\theta,a)$ and $v=(\theta,b)$.  Note that addition and scalar multiplication in $T_\theta \Theta$ effectively ignore the shared base-point $\theta$.

Similarly, we can view each tangent vector to the statistical model $\cM$ as a pair $(P,A)$, where the base-point $P$ is an element of the model $\cM$ and $A$ is essentially the score in a particular direction \cite[\S3.3]{PistoneSempri95}.  More precisely, for each tangent vector $u=(\theta,a)$ to $\Theta$, there is a corresponding tangent vector $\tilde{u} = (P,A)$ to $\cM$ given by
\begin{align}
\label{E:tvec_param}
P = p_\theta \mu \text{ and } A = \sum_{i=1}^d a_i \frac{\partial p_\theta}{\partial \theta_i} \mu.
\end{align}
(The function taking $u$ to $\tilde{u}$ is the differential of the parameterisation $\theta \mapsto p_\theta \mu$ \cite[Def. C.3.4]{KassVos97}.)  Let the tangent bundle $T\cM$ of $\cM$ be the set of all such tangent vectors, i.e., let $T\cM = \{ \tilde{u} \mid u \in T\Theta\}$.
Also, let the tangent space $T_P \cM$ to $\cM$ at $P \in \cM$ be the vector space consisting of all tangent vectors $(P,A) \in T \cM$ with base-point $P$.  Even though we have used a particular parameterisation of $\cM$ to define $T_P \cM$,  this tangent space is natural, in the sense that $T_P\cM$ is the same for all smooth parameterisations.

The Fisher information metric $g^F$ on $\cM$ is given by
\begin{align}
\label{E:FI_defn}
g^F(\tilde{u},\tilde{v}) = \int \frac{dA}{dP} \frac{dB}{dP} \; dP
\end{align}
for any tangent vectors $\tilde{u} = (P,A)$ and $\tilde{v} = (P,B)$ in the tangent space $T_P \cM$ \cite[\S3]{BauerEtAl16}, where $dA/dP$ and $dB/dP$ are Radon-Nikodym derivatives \citep[\S 3.2]{Bogachev07}.
It is straightforward (see Appendix \ref{S:E:FI_abs_conv}) to show that definition (\ref{E:FI_defn}) for the Fisher information metric reduces to the usual, parameterisation-dependent definition \citep[eq. 2.6]{AmariNagaoka00}.  However, the formulation (\ref{E:FI_defn}) will be more useful to us than the usual definition.  Also, because (\ref{E:FI_defn}) is phrased only in terms of natural constructions, this formula makes it clear that $g^F$ does not depend on arbitrary choices, such as the choice of parameterisation.

A Riemannian metric on a set is just a function that puts an inner product on each of the set's tangent spaces (if the set is suitably regular and the inner products vary smoothly with the base-point).  For example, a Riemannian metric on $\Theta$ can be thought of as a smooth, matrix-valued function on $\Theta$ whose value at $\theta \in \Theta$ is a $d \times d$, symmetric, positive definite matrix $\bar{g}_\theta$, since this defines an inner product on each $T_\theta \Theta$ with the inner product of any $u, v \in T_\theta \Theta$ being $g(u,v) = a^T \bar{g}_\theta b$, where $u=(\theta,a)$ and $v=(\theta,b)$.

In our main case of interest, where $\cM$ is an exponential family, the integral in (\ref{E:FI_defn}) always converges \citep[Thm. 2.2.5]{KassVos97}.  Then it is not hard to see that (\ref{E:FI_defn}) defines an inner product on each tangent space to $\cM$ (and this varies smoothly with the base-point), so the Fisher information metric $g^F$ is a Riemannian metric on $\cM$.

\section{Exponential families and their derived families}
\label{S:expfam}

Partly to establish our notation, this section briefly recalls the definitions of an exponential family, its IID extensions and their corresponding natural exponential families.

\subsection{Exponential families}

Let $\mu$ be a measure on $\R^m$ and let $T:\cX \to \R^d$ be a measurable function, where $\cX \subseteq \R^m$ is the support of $\mu$.  Let
$$\Theta = \left\{ \theta \in \R^d \; \left| \; \int  \exp(\theta \cdot T) d\mu < \infty \right. \right\},$$
where the dot ($\cdot$) denotes the Euclidean inner product on $\R^d$.  For each $\theta \in \Theta$, define $p_\theta:\cX \to \R_{>0}$ by
\begin{align}
\label{E:ptheta}
p_\theta(x) = \exp(\theta \cdot T(x))/Z(\theta)
\end{align}
for any $x \in \cX$, where $Z: \Theta \to \R$ is the partition function
$ Z(\theta) = \int  \exp(\theta \cdot T) d\mu$.  Assume that $\Theta$ is a non-empty, open subset of $\R^d$ and that $T$ is full rank, in the sense that the image of $T$ is not contained in any $(d-1)$-dimensional hyperplane in $\R^d$.  Then $\cM = \{ p_\theta \mu \mid \theta \in \Theta \}$ is a regular exponential family of order $d$ with dominating measure $\mu$ and canonical sufficient statistic $T$, and all regular exponential families are of this form \citep[\S8.1]{BarndorffNielsen78}.  Note that each element of $\cM$ is a probability measure on $\R^m$.

\subsection{IID extensions}
\label{S:IIDextensions}

The $n$-fold IID extension $\cM^n$ of $\cM$ is the set $\cM^n = \{ P^n \mid P \in \cM \}$ of all measures of the form $P^n$ for some $P \in \cM$, where $P^n = P \times \dots \times P$ (with $n$ copies of $P$) is the product measure on $\cX^n$ \citep[\S 3.3]{Bogachev07}.  In terms of the parameterisation (\ref{E:ptheta}), $\cM^n$ is the set of all measures of the form $p_\theta^{(n)} \mu^n$ for some $\theta \in \Theta$, where $p_\theta^{(n)}: \cX^n \to \R_{>0}$ is given by $p_\theta^{(n)}(x_1, \dots, x_n) = p_\theta(x_1) \dots p_\theta(x_n)$ and $\mu^n = \mu \times \dots \times \mu$ is the product measure on $\cX^n$ \citep[Example 8.12(ii)]{BarndorffNielsen78}.  So by (\ref{E:ptheta}),
\begin{align}
\label{E:pthetan}
\ifZ p_\theta^{(n)} = \exp(n\theta \cdot T_n - n \log Z(\theta)),
\else p_\theta^{(n)} = \exp(n\theta \cdot T_n - n \psi(\theta)),
\fi
\end{align}
where $T_n:\cX^n \to \R^d$ is given by $T_n(x_1, \dots, x_n) = (T(x_1) + \dots + T(x_n))/n$ for any $x_1, \dots, x_n \in \cX$.  Therefore $\cM^n$ is an exponential family with dominating measure $\mu^n$ and sufficient statistic $T_n$
(and natural parameter $n \theta$, see \cite[Thm. 2.2.6]{KassVos97}).
Note that $\cM^1 = \cM$, $T_1 = T$ and $p_\theta^{(1)} = p_\theta$.

\subsection{Natural exponential families}
\label{S:natexpfams}

Recall that if $\cY$ and $\cZ$ are measurable spaces, $\phi:\cY \to \cZ$ is a measurable function and $P$ is a measure on $\cY$ then the push-forward of $P$ via $\phi$ is the measure $\phi_*P$ on $\cZ$ given by
\begin{align}
\label{E:pushfwd}
(\phi_*P)(U) = P(\phi^{-1}(U))
\end{align}
for any measurable set $U$ in $\cZ$ \citep[\S 3.6]{Bogachev07}.
This immediately implies that if $Y$ is a $\cY$-valued random variable with distribution $P$ then $\phi(Y)$ is a $\cZ$-valued random variable with distribution $\phi_*P$, which in symbols we write as
\begin{align}
\label{E:pushfwd_rv}
Y \sim P \text{ implies } \phi(Y) \sim \phi_*P.
\end{align}

Then the natural exponential family corresponding to $\cM^n$ and $T_n$ is the set
$\cN_n = \{ {T_n}_*P^n \mid P^n \in \cM^n \}$
of measures on $\R^d$.
By \citep[Examples 8.12(ii) and 8.12(iii)]{BarndorffNielsen78}, $\cN_n = \{ q_\theta^n \nu_n \mid \theta \in \Theta\}$, where $\nu_n$ is a measure on $\R^d$ which does not depend on $\theta$ and $q_\theta^n: \R^d \to \R_{>0}$ is given by
\begin{align}
\label{E:qthetan}
\ifZ q_\theta^n(y) = \exp(n\theta \cdot y - n \log Z(\theta))
\else q_\theta^n(y) = \exp(n\theta \cdot y - n \psi(\theta))
\fi
\end{align}
for any $y \in \R^d$.
The formula (\ref{E:qthetan}) shows that the superscript in $q_\theta^n$ is actually an exponent, so we will write $q_\theta$ for $q_\theta^1$ (and then the notation $q_\theta^n$ is unambiguous).

Note that even though $\cM, \cM^2, \cM^3, \dots$ and $\cN_1, \cN_2, \cN_3, \dots$ are families of measures on different spaces (namely, $\cX, \cX^2, \cX^3, \dots$ and $\R^d, \R^d, \R^d, \dots$, respectively), they are all parameterised by $\Theta \subseteq \R^d$ so they are all $d$-dimensional families of measures.

\section{Invariance and regularity conditions}
\label{S:assumptions}

Let $\cM$, $\cM^n$ and $\cN_n$ be as in Section \ref{S:expfam} and suppose now that these spaces have been equipped with Riemannian metrics $g$, $g^n$ and $g_n$, respectively.  In this section, we will give precise conditions that formalize the notion of these metrics being invariant under IID extensions and canonical sufficient statistics, as well as giving a mild regularity condition.  These conditions will then be used in Section \ref{S:chentsov} to prove our main theorem.  See Section \ref{S:Assum:Remarks} for a number of remarks about these assumptions.

\begin{assumption}
We make the following assumptions, which are described precisely in the subsections below:
\begin{enumerate}
  \item[A1]  The metrics $g$ and $g^n$ are invariant under IID extensions (up to a factor of $n$)
  \item[A2]  The metrics $g^n$ and $g_n$ are invariant under canonical sufficient statistics
  \item[A3]  The norms corresponding to the metrics $g_n$ can all be calculated by a function that satisfies a weak continuity condition
\end{enumerate}
\end{assumption}

\subsection{A1: Invariance under IID extensions}
\label{S:Assum:IID}

Let $IID_n: \cM \to \cM^n$ be the function which maps each $P \in \cM$ to the product measure $P^n = P \times \dots \times P$ (see Section \ref{S:IIDextensions}). 
Then our first assumption is that this map is an isometry (i.e., distance-preserving map) up to a factor of $n$.

More precisely, let $u = (\theta, a) \in T\Theta$ be any tangent vector to $\Theta$, as in Section \ref{S:FImetric}.
Then similarly to (\ref{E:tvec_param}), $u$ corresponds under the smooth parameterisation (\ref{E:pthetan})
to a tangent vector $\tilde{u}^n$ to $\cM^n$, where $\tilde{u}^n = (P^n,A^{(n)})$, $P^n = p_\theta^{(n)} \mu^n$ and $A^{(n)} = \sum_{i=1}^d a_i (\partial p_\theta^{(n)}/\partial \theta_i) \mu^n$.  Let $T\cM^n = \{ \tilde{u}^n \mid u \in T\Theta\}$ be the set of all such tangent vectors to $\cM^n$.  Then our first assumption is that
\begin{align}
\label{E:gn_IID}
g^n(\tilde{u}^n, \tilde{v}^n) = n g(\tilde{u}, \tilde{v})
\end{align}
for all tangent vectors $u, v \in T \Theta$ with the same base-point.  Here, $\tilde{v}$ and $\tilde{v}^n$ are the tangent vectors to $\cM$ and $\cM^n$ (respectively) corresponding to $v \in T\Theta$, as for $u$ above.
Note that (\ref{E:gn_IID}) just says that $g^n = n g$ under the identification of $\cM$ with $\cM^n$ via $IID_n$.

The Fisher information metric is invariant under IID extensions in the sense of (\ref{E:gn_IID}) by \citep[eq. 4.2]{AmariNagaoka00}, so assumptions (A1)--(A3) cannot characterize the Fisher information metric unless the factor of $n$ is included in (\ref{E:gn_IID}) (though see Remark \ref{R:factorn}).

\subsection{A2: Invariance under canonical sufficient statistics}
\label{S:Assum:SS}

Let $T_n:\cX^n \to \R^d$ be the canonical sufficient statistic from Section \ref{S:IIDextensions} and let ${T_n}_*: \cM^n \to \cN_n$ be the corresponding (measure-theoretic) push-forward map of $T_n$, see Section \ref{S:natexpfams}.  Then our second assumption is that this map ${T_n}_*$ is an isometry (and that all other canonical sufficient statistics are isometries, in a sense which will be made precise in Section \ref{S:Assum:H}).

More precisely, let $u = (\theta, a) \in T\Theta$ be any tangent vector to $\Theta$, as in Section \ref{S:FImetric}.
Then similarly to (\ref{E:tvec_param}), $u$ corresponds under the smooth parameterisation (\ref{E:qthetan}) to a tangent vector $\tilde{u}_n = (Q_n,A_n)$ to $\cN_n$, where
\begin{align}
\label{E:u_n}
Q_n = q_\theta^n \nu_n \text{ and } A_n = \sum_{i=1}^d a_i (\partial q_\theta^n/\partial \theta_i) \nu_n.
\end{align}
Let $T\cN_n = \{ \tilde{u}_n \mid u \in T\Theta\}$ be the set of all such tangent vectors.
Then our second assumption is that
\begin{align}
\label{E:gn_suff}
g_n(\tilde{u}_n,\tilde{v}_n) = g^n(\tilde{u}^n,\tilde{v}^n)
\end{align}
for all tangent vectors $u, v \in T \Theta$ with the same base-point.  Here, $\tilde{v}^n$ and $\tilde{v}_n$ are the tangent vectors to $\cM^n$ and $\cN_n$ (respectively) corresponding to $v \in T\Theta$, as for $u$ above.
Note that (\ref{E:gn_suff}) just says that $g_n = g^n$ under the identification of $\cM^n$ with $\cN_n$ via ${T_n}_*$.

\subsection{A3: Calculability of norms by a function that satisfies a weak continuity condition}
\label{S:Assum:H}

Let $h$ be the norm corresponding to $g$, so $h(\tilde{u}) = \sqrt{g(\tilde{u},\tilde{u})}$ for any $\tilde{u} \in T \cM$.  Note that $h$ determines $g$ by the polarisation formula,
$$ g(\tilde{u},\tilde{v}) = \left[h^2(\tilde{u} + \tilde{v}) - h^2(\tilde{u} - \tilde{v})\right]/4 $$
for any $\tilde{u},\tilde{v} \in T\cM$ with the same base-point (which follows from the bilinearity of $g$), so any question about $g$ can be phrased in terms of $h$.  However, it will be more convenient to work with $h$ than $g$, because $h$ is a function defined on $T \cM$, whereas $g$ is only defined on certain pairs of tangent vectors (those with the same base-point).  Similarly, let $h_n$ be the norm corresponding to $g_n$, so $h_n(\tilde{u}_n) = \sqrt{g_n(\tilde{u}_n,\tilde{u}_n)}$ for any $\tilde{u}_n \in T \cN_n$.

Let $\cT^\prime$ be the set of all pairs $(P,A)$, where $P$ is a probability measure on $\R^d$ and $A$ is a signed measure on $\R^d$, and note that $T\cN_n \subseteq \cT^\prime$ for every $n$.
Then our regularity condition (A3) is, firstly, that there is subset $\cT$ of $\cT^\prime$ and a function $H: \cT \to \R$ so that, for each $n$, $T\cN_n \subseteq \cT$ (i.e. $H$ is defined on each $T\cN_n$) and
\begin{align}
\label{E:calcH}
h_n(\tilde{u}_n) = H(\tilde{u}_n)
\end{align}
for every $\tilde{u}_n \in T \cN_n$.  In other words, we assume that there is some function $H$ whose restriction to each $T \cN_n$ is the norm $h_n$.  For instance, we could take $\cT = \cup_{n = 1}^\infty T \cN_n$ and then define $H$ by the requirement that (\ref{E:calcH}) holds, which gives a well-defined $H$ whenever the functions $h_n$ agree on any overlaps between the spaces $T \cN_n$.

Further, we assume that $H$ has the following weak continuity property.  Firstly, we require that $H$ is defined on all pairs of the form $(\Phi, f \Phi)$, where $\Phi$ is the probability measure for the standard normal distribution on $\R^d$ and $f:\R^d \to \R$ is a linear function (with $f(0)=0$).  Secondly, we require that
\begin{align}
\label{E:contH}
H(P_n, f P_n) = H(\Phi, f \Phi)
\end{align}
for any sequence $P_n$ of probability measures on $\R^d$ for which $H(P_n, f P_n)$ is constant in $n$, $P_n \Rightarrow \Phi$ and each $P_n$ is standardized (i.e., $P_n$ has $0$ mean and identity variance-convariance matrix),
where $H(P_n, f P_n)$ is the value of the function $H$ at $(P_n, f P_n) \in \cT$ and
$P_n \Rightarrow \Phi$ means $P_n$ converges to $\Phi$ in the sense of the weak convergence of measures \cite[Def. 1.2.1]{MeerschaertScheffler01}.  This condition is an extremely weak form of continuity, see Remark \ref{R:regcond}.

Lastly, as a consequence of our assumption (A2) that the metrics should be invariant under all canonical sufficient statistics, we assume that $H$ is affine invariant (see Remark \ref{R:affineinvar}).  Here, an invertible affine transformation of $\R^d$ is a map $L:\R^d \to \R^d$ of the form $L(x) = Mx + c$ for some invertible $d\times d$ matrix $M$ and some $c \in \R^d$.  The push-forward $L_*A$ of any signed measure $A$ on $\R^d$ is defined in a similar way to the push-forward of an (unsigned) measure, see (\ref{E:pushfwd}).
We define the push-forward $L_{**}(P,A)$ of any $(P,A) \in \cT$ to be $L_{**}(P,A)=(L_*P,L_*A)$.  (In this notation, $L_*$ is the measure-theoretic push-forward, which is a map from the space of signed measures on $\R^d$ to itself, and $L_{**}$ is the differential of this map if $(P,A)$ is interpreted as a tangent vector.)
Then our condition that $H$ is affine invariant means that $L_{**}(P,A) \in \cT$ and
\begin{align}
\label{E:affH}
H(L_{**}(P,A) = H(P,A)
\end{align}
for every $(P,A)\in \cT$ and every invertible affine transformation $L$ of $\R^d$.

For future reference, we note that if $L$ is an invertible affine transformation, $P$ is a probability measure and $f$ is a $P$-integrable, real-valued function then
\begin{align}
L_*(fP) = (f \circ L^{-1}) L_*P
\label{E:ChOfVar}
\end{align}
by the change of variables formula \cite[Thm. 3.6.1]{Bogachev07}.

\subsection{Remarks on the assumptions}
\label{S:Assum:Remarks}

\begin{remark}
\label{R:regcond}
Assumptions (A1) and (A2) say that the metrics on $\cM$, $\cM^n$ and $\cN_n$ are invariant under a countable set of transformations and, in a certain sense, under the finite-dimensional group of affine transformations of $\R^d$.
The third assumption (A3) is an extremely weak form of continuity.  Firstly, this condition says that the norms $h_n$ agree on any overlaps between the spaces $T\cN_n$, so that these functions can be pieced together into a single function $H$.  Secondly, this condition says that if $f$ is linear and $P_n \Rightarrow \Phi$ is a sequence for which $(P_n,fP_n)$ all have the same norms then this shared norm must be $H(\Phi,f\Phi)$.  By comparison, full continuity of $H$
would require that $\lim_{n\to \infty} H(P_n,f_n P_n) = H(P,f P)$ for every sequence $(P_n,f_n P_n)$ in $\cT$ that converges to $(P,f P)$ (with respect to some notion of convergence).  So our third assumption is the condition for the continuity of $H$ in the very special case where $P=\Phi$, $H(P_n,f_n P_n)$ is constant in $n$, $f_n = f$ for every $n$ and $f$ is a linear function.
\end{remark}

\begin{remark}
\label{R:comparison}
Recent versions of Chentsov's theorem \citep{AyEtAl15,BauerEtAl16} consider metrics on infinite-dimensional statistical models that are invariant under infinite-dimensional sets of transformations.  This infinite dimensionality introduces technical complications and it makes strong assumptions about both the space on which the metric is defined and its symmetries.  By contrast, our approach allows us to only consider metrics on a collection of finite-dimensional models, as in the original version of Chentsov's theorem \citep{Chentsov78,Chentsov82,Campbell86}.  This allows our characterisation of the Fisher information metric to be relatively free from technicalities and it allows us to make relatively weak invariance and regularity assumptions.
\end{remark}

\begin{remark}
\label{R:FI}
It is not hard to see that the Fisher information metric satisfies assumptions (A1)--(A3).  For it is well known that the Fisher information metric is invariant under both IID extensions (in the sense of (\ref{E:gn_IID})) and sufficient statistics \citep[eq. 4.2 and Thm. 2.1]{AmariNagaoka00}.  Also, given any probability measure $P$ on $\R^d$, let $\cT_P = \{ (P,fP) \mid f \in L^2(\R^d, P) \}$, and let $\cT$ be the union of these spaces $\cT_P$ as $P$ ranges over the set of all probability measures on $\R^d$.  Then by (\ref{E:FI_defn}), the Fisher information norm $H^F(P,fP)$ of any $(P,fP) \in \cT$ is just the $L^2(\R^d, P)$-norm of $f$.  So if $f$ is a linear function on $\R^d$, say $f(y) = c \cdot y$ for some $c \in \R^d$, and $Q$ is any standardized probability measure on $\R^d$ then
$$ H^F(Q,fQ) = \sqrt{\int (c \cdot y)^2 dQ(y)} = \sqrt{c^T \left(\int y y^T dQ(y)\right) c} =\sqrt{c^T I c} = \| c\|, $$
where $\| c\|$ is the Euclidean norm of $c \in \R^d$.
So for any sequence $P_n$ of standardized probability measures (whether weakly convergent to $\Phi$ or not), $H^F(P_n,fP_n) = \| c\| = H^F(\Phi,f\Phi)$, so $H^F$ satisfies the weak continuity condition (\ref{E:contH}).  Lastly, this function $H^F$ is affine invariant (\ref{E:affH}) by the change of variables formula (\ref{E:ChOfVar}).
\end{remark}

\begin{remark}
\label{R:factorn}
In some ways the factor of $n$ in (\ref{E:gn_IID}) is not essential, since we could instead formulate our assumptions and theorems in terms of the metrics $\dot{g}^n = g^n/n$ and $\dot{g}_n = g_n/n$, in which case (\ref{E:gn_IID}) would be equivalent to the equation that describes exact invariance under the map $IID_n$, rather than invariance up to a factor of $n$ (though $H$ as in (\ref{E:calcH}) might not exist without the factor of $n$).  However, it is natural to include the factor of $n$ in our formulation of IID invariance, firstly because the Fisher information metric is IID invariant in the sense of (\ref{E:gn_IID}) \citep[eq. 4.2]{AmariNagaoka00},
so assumptions (A1)--(A3) would not characterise the Fisher information metric without this factor,
and secondly because the factor of $n$
arises from a natural construction from differential geometry (see Remark \ref{R:gn}).
\end{remark}

\begin{remark}
\label{R:gn}
Given an arbitrary Riemannian metric $g$ on $\cM$, a natural construction from differential geometry gives a metric on the $n$-fold IID extension $\cM^n$ of $\cM$ equal to the metric $g^n$ satisfying (\ref{E:gn_IID}), as follows.  The Cartesian product $\cartMn$ of $\cM$ with itself $n$ times is the space whose points are $n$-tuples $(P_1, \dots, P_n)$ of measures $P_1, \dots, P_n \in \cM$ on $\cX$.  Given such an $n$-tuple, there is a corresponding product measure $P_1 \times \dots \times P_n$ on $\cX^n$, and conversely we can recover each $P_i$ from $P_1 \times \dots \times P_n$ by marginalizing, so we can identify $(P_1, \dots, P_n)$ with the product measure $P_1 \times \dots \times P_n$ on $\cX^n$.  This product measure is the joint distribution of independent random variables $X_1, \dots, X_n$ whose marginal distributions are $P_1, \dots, P_n$, respectively.
So if $(P_1, \dots, P_n)\in \cartMn$ satisfies $P_1 = \dots = P_n$ then $P_1 \times \dots \times P_n$ is the joint distribution of IID random variables $X_1, \dots, X_n$.
Therefore we can identify the diagonal
$$ \Delta = \left\{ \left. (P_1, \dots, P_n)\in \cartMn \right| P_1 = \dots = P_n \right\} $$
of $\cartMn$ with the $n$-fold IID extension $\cM^n$ of $\cM$.
But a Riemannian metric on $\cM$ induces a Riemannian metric on the Cartesian product $\cartMn$, and then $\Delta$ inherits a metric from its super-manifold $\cartMn$.  Under the above identification between $\Delta$ and $\cM^n$, this metric is the metric $g^n$ on $\cM^n$ that satisfies (\ref{E:gn_IID}).
\end{remark}

\begin{remark}
\label{R:affineinvar}
The canonical sufficient statistics for an exponential family are only unique up to affine transformations
\citep[Lemma 8.1]{BarndorffNielsen78}, meaning that if $L$ is an invertible affine transformation of $\R^d$ and $T_n:\cX^n \to \R^d$ is a canonical sufficient statistic then $L \circ T_n$ is also a canonical sufficient statistic (and every canonical sufficient statistic is of this form).
Replacing $T_n$ by $L \circ T_n$ effectively replaces each tangent vector $\tilde{u}_n \in T \cN_n$ by $L_{**} \tilde{u}_n$, so (\ref{E:gn_suff}), (\ref{E:calcH}) and the analogous equations for $L \circ T_n$ imply $H(L_{**} \tilde{u}_n) = H(\tilde{u}_n)$ for every $\tilde{u}_n \in T \cN_n$.
So since $L$ is arbitrary, $H$ is affine invariant.
\end{remark}

\section{The main theorem}
\label{S:chentsov}

We can now prove our version of Chentsov's theorem.  This theorem characterises the Fisher information metric as the only metric (up to rescaling) on an exponential family that is invariant under IID extensions and canonical sufficient statistics.

Let $g^F$, $g^{nF}$ and $g_n^F$ be the Fisher information metrics on $\cM$, $\cM^n$ and $\cN_n$, respectively.

\begin{theorem}
\label{T:main}
Suppose that assumptions (A1)--(A3) of Section \ref{S:assumptions} hold.
Then there is some $c>0$ so that $g = c g^F$, $g^n = c g^{nF}$ and $g_n = c g_n^F$ for every integer $n \ge 1$.
\end{theorem}

\begin{proof}
Let any integer $n \ge 1$ and any $\theta \in \Theta$ be given, and let $Q_1 = q_\theta \nu_1 \in \cN_1$ and $Q_n = q_\theta^n \nu_n \in \cN_n$ be the corresponding distributions in $\cN_1$ and $\cN_n$.  By Theorem 2.2.6 of \citep{KassVos97} and the comments preceding it, if $Y_1, \dots, Y_n$ are independent random variables
all distributed according to $Q_1$ then their mean is distributed as $Q_n$, which we write as
\begin{align}
\label{E:Pn_is_av}
(Y_1 + \dots + Y_n)/n \sim Q_n.
\end{align}
Alternatively, it is not hard to prove (\ref{E:Pn_is_av}), since if $X_1, \dots , X_n \sim p_\theta \mu$ are IID and $Y_i^\prime = T(X_i)$ then $Y_1^\prime, \dots, Y_n^\prime \sim Q_1$ are IID and $(Y_1^\prime + \dots + Y_n^\prime)/n = T_n(X_1, \dots , X_n) \sim Q_n$, by (\ref{E:pushfwd_rv}) and
since $Q_1 = T_*P$ and $Q_n = T_{n*}P^n$ (by definition), where $P = p_\theta \mu$.
This proves (\ref{E:Pn_is_av}) because
$Y_1, \dots, Y_n$ and $Y_1^\prime, \dots, Y_n^\prime$ have the same joint distribution so their means have the same distribution, by another application of (\ref{E:pushfwd_rv}).

By (\ref{E:Pn_is_av}),
the mean $\meansuff$ for $Q_1$ is the same as that for $Q_n$, i.e.
\begin{align}
\label{E:mean_suff}
\meansuff = \int y dQ_1(y) = \int y dQ_n(y),
\end{align}
and the variance-covariance matrix $\Sigma_\theta$ for $Q_1$ is $n$ times that for $Q_n$, i.e.
\begin{align}
\label{E:var_suff}
\Sigma_\theta = \int (y - \meansuff)(y - \meansuff)^T dQ_1(y) = n \int (y - \meansuff)(y - \meansuff)^T dQ_n(y).
\end{align}

Now, let $u = (\theta,a) \in T_\theta \Theta$ be any tangent vector to $\Theta$ at $\theta$, and define $f: \R^d \to \R$ by $f(y) = (\Sigma_\theta^{1/2} a)\cdot y$ for any $y \in \R^d$.  Here, $\Sigma_\theta^{1/2}$ is defined in the standard way via a diagonalisation of the symmetric, positive-definite matrix $\Sigma_\theta$.  As before, let $\tilde{u}$ and $\tilde{u}_n$, respectively, be the tangents to $\cM$ and $\cN_n$ that correspond to $u$ under the parameterisations (\ref{E:pthetan}) and (\ref{E:qthetan}).

{\em Claim 1: $h(\tilde{u}) = H(\Phi, f \Phi)$.}
By (\ref{E:qthetan}), (\ref{E:u_n}) and the fact that $\meansuff$ is the gradient of
\ifZ $\log Z$
\else $\psi$
\fi
at $\theta$ \cite[Thm. 2.2.1]{KassVos97}, $\tilde{u}_n = (Q_n,A_n)$ with $Q_n = q_\theta^n \nu_n$ and
\ifiota
\begin{align}
\label{E:An}
A_n = \sum_{i=1}^d a_i \frac{\partial q_\theta^n}{\partial \theta_i} \nu_n
\ifZ  = \sum_{i=1}^d a_i n \left(\iota_i - \frac{\partial \log Z}{\partial \theta_i}\right) q_\theta^n \nu_n
\else = \sum_{i=1}^d a_i n \left(\iota_i - \frac{\partial \psi}{\partial \theta_i}\right) q_\theta^n \nu_n
\fi
= na \cdot (\iota - \meansuff) Q_n,
\end{align}
where $\iota_i(y) = y_i$ and $\iota(y) = y$ for any $y \in \R^d$.
\else
$A_n$ is given by
\begin{align}
\label{E:An}
dA_n(y) = \sum_{i=1}^d a_i \frac{\partial q_\theta^n}{\partial \theta_i} d\nu_n(y)
\ifZ  = \sum_{i=1}^d a_i n \left(y_i - \frac{\partial \log Z}{\partial \theta_i}\right) q_\theta^n d\nu_n(y)
\else = \sum_{i=1}^d a_i n \left(y_i - \frac{\partial \psi}{\partial \theta_i}\right) q_\theta^n d\nu_n(y)
\fi
= na \cdot (y - \meansuff) dQ_n(y),
\end{align}
for any $y \in \R^d$.
\fi

Let $L$ be the affine transformation on $\R^d$ given by $L(y) = \sqrt{n} \Sigma_\theta^{-1/2} (y - \meansuff)$, and note that $\Sigma_\theta^{-1/2}$ exists because $\Sigma_\theta$ is positive-definite.
By (\ref{E:mean_suff}) and (\ref{E:var_suff}), this choice of $L$ ensures that $L_*Q_n$ is standardised, i.e., that $L_*Q_n$ has mean $0$ and variance-covariance matrix equal to the $d \times d$ identity matrix.
Note that $L$ depends on $n$, so we could instead write this as $L_n$, but for notational simplicity we will drop the subscript.
\ifiota
Then by (\ref{E:ChOfVar}) and (\ref{E:An}),
\begin{align}
L_*A_n
= na \cdot (\iota \circ L^{-1} - \meansuff) L_*Q_n
= \sqrt{n} f L_*Q_n,
\label{E:AnL}
\end{align}
where $f$ is as in the statement of the claim.
\else
Then by (\ref{E:ChOfVar}) and (\ref{E:An}), $L_*A_n$ is given by
$$ d(L_*A_n)(y)
= na \cdot (L^{-1}(y) - \meansuff) d(L_*Q_n)(y)
= (\sqrt{n} \Sigma_\theta^{1/2} a)\cdot y d(L_*Q_n)(y) $$
for any $y \in \R^d$, so
\begin{align}
L_*A_n = \sqrt{n} f L_*Q_n,
\label{E:AnL}
\end{align}
where $f(y) = (\Sigma_\theta^{1/2} a)\cdot y$ for any $y \in \R^d$.
\fi

So recalling the notation $L_{**}\tilde{u}_n = L_{**}(Q_n,A_n) = (L_*Q_n, L_*A_n)$, we have
\begin{align}
h(\tilde{u})
&= n^{-1/2} h_n(\tilde{u}_n) \text{ by (\ref{E:gn_IID}) and (\ref{E:gn_suff})} \nonumber \\
&= h_n(n^{-1/2} \tilde{u}_n) \text{ by the bilinearity of $g_n$} \nonumber \\
&= H(n^{-1/2} \tilde{u}_n) \text{ by (\ref{E:calcH})}  \nonumber \\
&= H(n^{-1/2} L_{**}\tilde{u}_n) \text{ by (\ref{E:affH})}  \nonumber \\
&= H(L_*Q_n, f L_*Q_n) \text{ by (\ref{E:vecsp}) and (\ref{E:AnL}).}  \label{E:step1}
\end{align}

By (\ref{E:Pn_is_av}), the central limit theorem (e.g. see \cite[Cor. 8.1.10]{MeerschaertScheffler01}) and the fact that $L_*Q_n$ is standardised, $L_*Q_n \Rightarrow \Phi$.  Therefore,
\begin{align}
h(\tilde{u})
&= H(L_*Q_n, f L_*Q_n) \text{ for all $n$, by (\ref{E:step1})}  \nonumber \\
&= H(\Phi, f \Phi) \text{ by (\ref{E:contH}),}  \label{E:useofcontcond}
\end{align}
so the claim is proved.

Now, let $v = (\phi,b) \in T\Theta$ be any tangent vector to $\Theta$, not necessarily with the same base-point as $u$, and let $\tilde{v} \in T\cM$ be the corresponding tangent vector to $\cM$.

{\em Claim 2: $a^T \Sigma_\theta a = b^T \Sigma_\phi b$ implies $h(\tilde{u})=h(\tilde{v})$.}
To prove this, assume that $a^T \Sigma_\theta a = b^T \Sigma_\phi b$, i.e. that $\Sigma_\theta^{1/2} a$ and $\Sigma_\phi^{1/2} b$ have the same Euclidean norm.  Then there exists a $d \times d$ orthogonal matrix $M$ so that
\begin{align}
\label{E:M}
M \Sigma_\theta^{1/2} a = \Sigma_\phi^{1/2} b.
\end{align}
Also, $M_*\Phi = \Phi$ because $M$ is orthogonal, so
\begin{align}
M_{**}(\Phi,f\Phi)=
(M_*\Phi,M_*(f\Phi))
= (M_*\Phi,(f\circ M^{-1}) M_*\Phi)
= (\Phi,e \Phi)
\label{E:Mstar}
\end{align}
by (\ref{E:ChOfVar}), where $e:\R^d \to \R$ is given by
\begin{align}
e(y) = f(M^{-1}(y)) = (\Sigma_\theta^{1/2} a)\cdot M^{-1}y
= (\Sigma_\theta^{1/2} a)^T M^{-1} y
= (\Sigma_\phi^{1/2} b) \cdot y
\label{E:e_defn}
\end{align}
for any $y \in \R^d$, by (\ref{E:M}) and $M^{-1}=M^{T}$ (since $M$ is orthogonal).  So
\begin{align}
h(\tilde{v})
&= H(\Phi, e \Phi) \text{ by Claim 1 applied to $v$ and by (\ref{E:e_defn})}  \nonumber \\
&= H(M_{**}(\Phi,f\Phi) \text{ by (\ref{E:Mstar})}  \nonumber \\
&= H(\Phi,f\Phi) \text{ by (\ref{E:affH})}  \nonumber \\
&= h(\tilde{u}) \text{ by Claim 1,} \nonumber 
\end{align}
which proves Claim 2.

{\em Claim 3: There is some $c > 0$ so that $h(\tilde{v}) = c \; h^F(\tilde{v})$ for all tangent vectors $\tilde{v} \in T \cM$. }
It is well-known \cite[Thms. 2.2.1 and 2.2.5]{KassVos97} that the Fisher information metric on the natural parameter space is the variance-covariance matrix of the corresponding sufficient statistic, so
$g^F(\tilde{u},\tilde{u}) = a^T \Sigma_\theta a$.  Alternatively, this follows easily from setting $n=1$ in (\ref{E:An}) and combining this with (\ref{E:FI_defn}) and the invariance of $g^F$ under sufficient statistics \citep[Thm. 2.1]{AmariNagaoka00}, since these give
\begin{align}
\label{E:FI_nat}
g^F(\tilde{u},\tilde{u}) = g^F_1(\tilde{u}_1,\tilde{u}_1)
= a^T \left(\int (y - \meansuff)(y - \meansuff)^T dQ_1(y) \right) a
= a^T \Sigma_\theta a,
\end{align}
where $\tilde{u}_1 \in T\cN_1$ is the tangent vector to $\cN_1$ corresponding to $u \in T\Theta$.
So Claim 2 is equivalent to
\begin{align}
\label{E:claim2rephrased}
h^F(\tilde{u}) = h^F(\tilde{v}) \text{ implies } h(\tilde{u}) = h(\tilde{v}),
\end{align}
for all tangent vectors $\tilde{u},\tilde{v} \in T \cM$, even if they have different base-points.

Now, fix $\tilde{u}$ to be some non-zero vector with $h^F(\tilde{u}) = 1$, and let $c = h(\tilde{u})$.  Note that $c>0$ because $g$ is an inner product on each tangent space so the norm of any non-zero tangent vector
is strictly positive.  Then for any non-zero $\tilde{v}$, $h^F(\tilde{v}/h^F(\tilde{v})) = h^F(\tilde{v})/h^F(\tilde{v}) =1$ by the bilinearity of $g^F$.  So $h^F(\tilde{u}) = h^F(\tilde{v}/h^F(\tilde{v}))$ and hence, by (\ref{E:claim2rephrased}), $h(\tilde{u}) = h(\tilde{v}/h^F(\tilde{v}))$.  Therefore $c = h(\tilde{u}) = h(\tilde{v}/h^F(\tilde{v}))= h(\tilde{v})/h^F(\tilde{v})$ by the bilinearity of $g$, so rearranging this equation proves the claim for all non-zero tangent vectors $\tilde{v} \in T \cM$.  But the claim holds trivially for any zero tangent vector $\tilde{v}$, since
$0 = h(\tilde{v}) = h^F(\tilde{v})$ by the bilinearity of $g$ and $g^F$, so the claim is proved.

The theorem now follows from Claim 3 and by (\ref{E:gn_IID}), (\ref{E:gn_suff}) and the analogous equations for the Fisher information metrics $g^F$, $g^{nF}$ and $g_n^F$, which hold by \citep[eq. 4.2 and Thm. 2.1]{AmariNagaoka00}.
\end{proof}

\section{Extensions to higher-order symmetric tensors}
\label{S:higherorder}

The proof of Theorem \ref{T:main} extends with almost no changes to characterise symmetric, order-$k$ tensors $\gk$ and $\gkn$ on $\cM$ and $\cN_n$, respectively, that satisfy conditions closely analogous to assumptions (A1)--(A3) of Section \ref{S:assumptions}.  Given such tensors $\gkn$, define $\hkn(\tilde{u}_n) = \sqrt[k]{\gkn(\tilde{u}_n, \dots, \tilde{u}_n)}$, where there are $k$ copies of $\tilde{u}_n$ in the right-hand side of this equation.  Assume that
\begin{align}
\label{E:higher_powerofn}
\gkn(\tilde{u}_n, \dots, \tilde{u}_n) = n^{k/2} \gkone(\tilde{u}_1, \dots, \tilde{u}_1),
\end{align}
which is a generalisation of (\ref{E:gn_IID}) from $k=2$ to general $k$.
Then as in the proof of Theorem \ref{T:main}, $\hkn(\tilde{u}_n) = \sqrt{n} \hkone(\tilde{u}_1)$ and $\hkn(\alpha \tilde{u}_n) = \alpha \hkn(\tilde{u}_n)$ for any $\alpha \ge 0$ (by (\ref{E:higher_powerofn}) and the multi-linearity of $\gkn$).  So with $\hk$ in place of $h$, the proof of Theorem \ref{T:main} implies that $\hk(\tilde{u}) = c \; h^F(\tilde{u})$ for some $c \in \R$, where $h^F$ is the norm of the Fisher information metric.  Raising this equation to the power of $k$ gives
\begin{align}
\label{E:higher}
\gk(\tilde{u}, \dots, \tilde{u}) = c^{k} \left[g^F(\tilde{u}, \tilde{u}) \right]^{k/2}.
\end{align}
If $k$ is odd then the left-hand side is an odd function of $\tilde{u}$ (i.e. it changes sign when $\tilde{u}$ is replaced by $-\tilde{u}$) while the right-hand side is an even function, which is a contradiction unless both sides vanish, so $c=0$.  If $k$ is even, then since $\gk$ is determined by (\ref{E:higher}) (by the polarisation formula for symmetric tensors), $\gk$ must be a constant times the symmetric part of $(g^F)^{k/2}$.  For example, when $k=4$ then there is some $c^\prime \in \R$ so that
$$ \gk(\tilde{u}, \tilde{v}, \tilde{w}, \tilde{m}) = c^\prime \left[ g^F(\tilde{u},\tilde{v})g^F(\tilde{w},\tilde{m}) + g^F(\tilde{u},\tilde{w})g^F(\tilde{v},\tilde{m}) + g^F(\tilde{u},\tilde{m})g^F(\tilde{v},\tilde{w})  \right] $$
for any $\tilde{u}, \tilde{v}, \tilde{w}, \tilde{m} \in T\cM$.

\begin{remark}
It might also be possible to adapt the proof of Theorem \ref{T:main} to characterise the higher-order Amari-Chentsov tensors, which are symmetric, order-$k$ tensors that coincide with the Fisher information metric when $k=2$ and in general are given by an equation similar to (\ref{E:FI_defn}), e.g. see \cite[eq. 2.4]{AyEtAl15} for the $k=3$ case.
Claim 1 in the proof of Theorem \ref{T:main} does not seem to hold for these tensors in general.  However, if we replace the $k/2$ in (\ref{E:higher_powerofn}) by other powers and strengthen the weak continuity condition on $H$ then
it might be possible to replace Claim 1 by
$\hk(\tilde{u}) = H(K \Phi, f K \Phi)$,
where $K$ is an Edgeworth polynomial (see \citep{BarndorffNielsenCox79} or \citep[\S4.5]{KassVos97}).
Then a symmetry argument, similar to the one in the proof of Theorem \ref{T:main}, should give the desired characterisation.
\end{remark}

\section{Discussion}
\label{S:conclusion}

Our version of Chentsov's theorem characterises the Fisher information metric as the unique Riemannian metric (up to rescaling) on an exponential family $\cM$ which is invariant under IID extensions and canonical sufficient statistics.
We proved this by considering metrics $g$ on $\cM$, $g^n$ on the $n$-fold IID extension $\cM^n$ of $\cM$, and $g_n$ on the natural exponential family $\cN_n$ corresponding to $\cM^n$.  Then, under the above invariance conditions, $g$ can be calculated in terms of $g_n$, for any $n$.
But for large $n$, the central limit theorem and a property (\ref{E:Pn_is_av}) of exponential families imply that $\cN_n$ consists of distributions which are all approximately normally distributed, so each distribution in $\cN_n$ is determined to a good approximation by its mean and variance-covariance matrix.  Further, each tangent vector to $\cN_n$ is essentially a linear function $f$ times a distribution in $\cN_n$.  Combining these facts shows that (the norm corresponding to) $g$ is approximately equal to a simple function of 
$f$ and the mean and variance-covariance matrix of the relevant distribution in $\cN_n$.  Our regularity condition implies that this approximation becomes exact in the limit as $n \to \infty$. 
Then our main result follows from an identity (\ref{E:FI_nat}) relating the variance-covariance matrix
to the Fisher information metric on an exponential family.

In general, Chentsov's theorem characterizes the Fisher information metrics of statistical models as the only Riemannian metrics (up to rescaling) that are invariant under certain, statistically important transformations.  Previous studies have taken these transformations to be either all sufficient statistics or a large, regular subset of these.  By contrast, we take these statistically important transformations to be the IID extensions and canonical sufficient statistics.  This class of transformations is arguably more natural than the class of all sufficient statistics, it is more appropriate for exponential families and it is a relatively small class so our invariance assumptions are weaker than those of previous studies.  Our regularity assumptions also appear to be weaker than previous studies, ultimately due to the fact that our approach only requires us to study a collection of finite-dimensional models, rather than an infinite-dimensional model.

We have given a new characterisation of the Fisher information metric on an exponential family and we have shown that this result is an intuitive consequence of the central limit theorem.
The main limitation of this paper is that our main result is only proved for exponential families.
However, exponential families are an important class of statistical models, being well studied and widely used in applications.  Also, our proof treats discrete and continuous models in a uniform way, so there is some hope that our approach can be adapted to give a proof of Chentsov's theorem for general statistical models.  Lastly, our focus on exponential families complements the focus of \citet{BauerEtAl16} on diffeomorphism-invariant metrics, since (curved) exponential families are essentially the only statistical models which have smooth sufficient statistics that are not diffeomorphisms, by the Pitman--Koopman--Darmois theorem \citep{BarndorffNielsenPedersen68}.

\appendix

\section{The invariant and parameterisation-dependent definitions of the Fisher information metric coincide}
\label{S:E:FI_abs_conv}

This section proves (in the notation of Section \ref{S:FImetric}) that the invariant definition (\ref{E:FI_defn}) of the Fisher information metric reduces to the usual parameterisation-dependent definition given by (\ref{E:FI_defn_conv}), below.

Given any tangent vectors $u=(\theta,a)$ and $v=(\theta,b)$ in $T_\theta \Theta$, let $\tilde{u} = (P,A)$ and $\tilde{v} = (P,B)$ be the corresponding tangent vectors in $T_P\cM$, where $P=p_\theta \mu$.  Then by (\ref{E:tvec_param}),
$$ A = \sum_{i=1}^d a_i \frac{\partial p_\theta}{\partial \theta_i} \mu
= \sum_{i=1}^d \frac{a_i}{p_\theta} \frac{\partial p_\theta}{\partial \theta_i} p_\theta \mu
= \sum_{i=1}^d a_i \frac{\partial \log p_\theta}{\partial \theta_i} P$$
so $dA/dP = \sum_{i=1}^d a_i (\partial/\partial \theta_i) \log p_\theta$, and similarly for $dB/dP$.  Substituting these into (\ref{E:FI_defn}) gives
\begin{align}
\label{E:FI_abs_conv}
g^F(\tilde{u},\tilde{v}) = \sum_{i,j=1}^d a_i b_j \int \left(\frac{\partial \log p_\theta }{\partial \theta_i} \right) \left( \frac{\partial \log p_\theta }{\partial \theta_j} \right) p_\theta d\mu
= a^T \bar{g}^{F}_\theta b,
\end{align}
where $\bar{g}^{F}_\theta$ is the $d \times d$ matrix with $(i,j)^{th}$ entry
\begin{align}
\label{E:FI_defn_conv}
[\bar{g}^{F}_\theta]_{ij} = \int \left(\frac{\partial \log p_\theta }{\partial \theta_i} \right) \left(
\frac{\partial \log p_\theta }{\partial \theta_j} \right) p_\theta d\mu,
\end{align}
for any $i,j = 1, \dots, d$.  Therefore the invariant definition (\ref{E:FI_defn}) reduces to the usual, parameterisation-dependent definition (\ref{E:FI_defn_conv}) for the Fisher information metric \citep[eq. 2.6]{AmariNagaoka00}.

\begin{remark}
The metric $\bar{g}^{F}$ on $\Theta$ is just the pull-back of the metric $g^F$ on $\cM$ via the parameterisation map $\Theta \to \cM$.
\end{remark}

\bibliographystyle{abbrvnat}
\bibliography{R:/5050/CEB/Share/Staff/EnesAndDaniel/Bibliography/bibliography}

\end{document}